\documentclass[12pt]{article}
\usepackage{amsmath, amssymb,latexsym}
\usepackage{color}
\usepackage{hyperref}
\title{Semigroup expansions for non-selfadjoint Schr\"odinger operators}
\author{Ben Bellis \\Department of Mathematics \\University of Michigan \\ Ann Arbor \\MI 48109, USA \\\small bbellis@umich.edu \and
Michael Hitrik\\Department of Mathematics \\University of California \\ Los Angeles
\\CA 90095-1555, USA\\\small hitrik@math.ucla.edu}
\date{}

\def\wrtext#1{\relax\ifmmode{\leavevmode\hbox{#1}}\else{#1}\fi}
\def\abs#1{\left|#1\right|}
\def\begeq{\begin{equation}}
\def\endeq{\end{equation}}

\def\part#1{\frac{\partial}{\partial #1}}

\def\norm#1{||\,#1\,||}

\newcommand{\real}{\mbox{\bf R}}
\newcommand{\comp}{\mbox{\bf C}}

\newcommand{\nat}{\mbox{\bf N}}

\renewcommand{\exp}{\mbox{\rm exp\,}}

\newtheorem{dref}{Definition}[section]
\newtheorem{lemma}[dref]{Lemma}
\newtheorem{theo}[dref]{Theorem}

\newenvironment{proof}{\vspace{.3cm}\noindent{{\em Proof:}}}{\hfill$\Box$}

\begin{document}
\maketitle

\vspace*{1cm}
\noindent
{\bf Abstract}: A large time expansion for the propagator associated to a semiclassical non-selfadjoint magnetic Schr\"odinger operator is established, in terms of the low lying eigenvalues of the operator.



\tableofcontents
\section{Introduction and statement of result}
\setcounter{equation}{0}
The purpose of this note is to apply the spectral results of~\cite{HiPr2},~\cite{HiPr3} together with the resolvent bounds obtained in ~\cite{B1},~\cite{B2}, to establish an expansion for the evolution semigroup associated to a class of semiclassical non-selfadjoint magnetic Schr\"odinger operators, in the limit of large times. It is well known that the exponential decay of a contraction semigroup on a Hilbert space is closely connected to the resolvent estimates for the corresponding semigroup generator, see~\cite{HeSj10},~\cite{Nier14}, and the references given there. Restricting the attention to the case of generators given by semiclassical pseudodifferential operators on $\real^n$, relevant here, let us recall roughly some general resolvent bounds provided by the abstract operator theory, following~\cite{DeSjZw},~\cite{Viola10}. Let $P = p^w(x,hD_x)$ be the semiclassical Weyl quantization of a complex-valued symbol $p\in C^{\infty}(\real^{2n}_{x,\xi})$ with ${\rm Re}\, p\geq 0$, belonging to a suitable symbol class and satisfying ellipticity conditions at infinity in $(x,\xi)$. Then we have
\begeq
\label{eq0.1}
\norm{(z-P)^{-1}}_{{\cal L}(L^2({\bf R}^n), L^2({\bf R}^n))}\leq \exp({\cal O}(1) h^{-n}) \prod_{\lambda \in {\rm Spec}(P)\cap \widetilde{\Omega}} \abs{z-\lambda}^{-1}, \quad
\endeq
for $z\in \Omega \subset \subset \widetilde{\Omega}\subset \subset \comp$. Under mild additional assumptions, guaranteeing the maximal accretivity of $P$, we may define the evolution semigroup $e^{-tP/h}$, $t\geq 0$, and represent it as the inverse Laplace transform of the resolvent, via a Cauchy integral formula,
\begeq
\label{eq0.2}
e^{-tP/h} = \frac{1}{2\pi i} \int_{\gamma} e^{-tz/h} (z-P)^{-1}\, dz.
\endeq
Here $\gamma = a + i\real$, $a<0$. We would like to obtain an expansion for $e^{-tP/h}$ by pushing the contour of integration in (\ref{eq0.2}) to the right, and here (\ref{eq0.1}) suggests that, if valid, such an expansion may only be of interest for $t\gg h^{-n}$. Now sharper resolvent estimates have been derived for non-selfadjoint $h$-pseudodifferential operators, under non-trapping and subellipticity conditions, when the spectral parameter varies in $h$-dependent regions close to the boundary of the semiclassical pseudospectrum $\overline{p(\real^{2n})}$. See~\cite{DeSjZw},~\cite{Sj10},~\cite{Sj_book}, and also~\cite{HeSjSt}. The recent results of~\cite{B1},~\cite{B2} establish polynomial resolvent estimates for a broad class of non-selfadjoint semiclassical magnetic Schr\"odinger operators in an unbounded parabolic region near the imaginary axis, and these works have been the starting point for this note. Let us now proceed to describe the assumptions and state the main results.

\bigskip
\noindent
Let us consider
\begeq
\label{eq1.1}
p(x,\xi) = \left(\xi - A(x)\right)^2 + V(x),\quad (x,\xi) \in \real^{2n}.
\endeq
Here $V = V_1 + iV_2$, $V_1,V_2\in C^{\infty}(\real^n;\real)$, is a complex-valued electric potential and $A = (A_1,\ldots\, A_n) \in C^{\infty}(\real^n, \real^n)$ is a real-valued magnetic potential. We shall assume that
\begeq
\label{eq1.2}
V_1 \geq 0.
\endeq
We shall furthermore impose the following growth conditions on $A_j$, $1\leq j \leq n$, and $V$,
\begeq
\label{eq1.3}
\nabla A_j(x) = {\cal O}(1),
\endeq
\begeq
\label{eq1.4}
\partial^{\alpha} A_j(x) = {\cal O}\left(\langle{x\rangle}^{-1}\right), \quad \abs{\alpha}\geq 2,
\endeq
\begeq
\label{eq1.5}
\partial^{\alpha} V(x) = {\cal O}(1), \quad \abs{\alpha}\geq 2.
\endeq

\medskip
\noindent
Following~\cite{B1},~\cite{B2}, we shall also assume that
\begeq
\label{eq1.6}
\abs{V_2(x)} \leq {\cal O}(1)\left(1 + V_1(x) + \abs{V_2'(x)}^2 \right).
\endeq

\medskip
\noindent
Associated to the symbol $p$ in (\ref{eq1.1}) is the magnetic Schr\"odinger operator given as the semiclassical Weyl quantization of $p$,
\begeq
\label{eq1.7}
P = p^w(x,hD_x) = \left(hD_x - A(x)\right)^2 + V(x), \quad D_x = i^{-1} \partial_x,\quad 0 < h \leq 1.
\endeq
From ~\cite[Section 6]{B2} let us recall that under the assumptions (\ref{eq1.3}) -- (\ref{eq1.5}), the $L^2$-graph closure of $P$ on the Schwartz space ${\cal S}(\real^n)$, still denoted by $P$, agrees with the maximal closed realization of $P$, with the domain
\begeq
\label{eq1.8}
{\cal D}(P) = \{u\in L^2(\real^n); Pu \in L^2(\real^n)\}.
\endeq
We have ${\rm Re}\, P \geq 0$ in view of (\ref{eq1.2}), and it follows that when equipped with the domain (\ref{eq1.8}), the operator $P$ is maximally accretive. An application of the Hille-Yosida theorem allows us to conclude that the strongly continuous contraction semigroup
\begeq
\label{eq1.9}
e^{-tP/h}: L^2(\real^n) \rightarrow L^2(\real^n),\quad t \geq 0,
\endeq
is well defined.

\medskip
\noindent
{\it Remark}. The maximal accretivity of the magnetic Schr\"odinger operator $P$ in (\ref{eq1.7}) holds in great generality, assuming only that (\ref{eq1.2}) holds, see~\cite{HeNo} and references given there.

\medskip
\noindent
In this note, we shall be concerned with the behavior of the evolution semigroup (\ref{eq1.9}) in the limit of large times. We shall be particularly interested in the relation between the exponential decay of the semigroup and the spectrum of $P$ near the imaginary axis.

\medskip
\noindent
To this end, let us now proceed to introduce some additional assumptions, which will allow us to obtain a precise description of the spectrum of $P$ in a disc around the origin of radius ${\cal O}(h)$, and an estimate for the resolvent of $P$ in the disc, as well as further away from this set, as a consequence of the results of~\cite{HiPr2},~\cite{HiPr3},~\cite{B2}. See also~\cite{HeSjSt}. Let us assume that there exists $C > 1$ such that
\begeq
\label{eq1.10}
V_1(x) \geq \frac{1}{C} \left(1 + \abs{V_2'(x)}^2\right),\quad \abs{x}\geq C.
\endeq
We assume furthermore that the set $V_1^{-1}(0) \subset \real^n$ is finite, so that
\begeq
\label{eq1.11}
V_1^{-1}(0) = \left\{x_1,\ldots\, x_N\right\}.
\endeq
In view of (\ref{eq1.2}) we see that $\nabla V_1(x_j) = 0$, $1\leq j \leq N$, and we shall assume that the imaginary part of the potential also vanishes to the second order along $V_1^{-1}(0)$,
\begeq
\label{eq1.12}
V_2(x_j) = \nabla V_2(x_j) = 0,\quad 1 \leq j \leq N.
\endeq
The Taylor expansion of $V$ at $x_j$, $1\leq j \leq N$, then has the form
\begeq
\label{eq1.13}
V(x_j + y) = \frac{1}{2} V''(x_j)y\cdot y + {\cal O}(y^3),\quad y\rightarrow 0.
\endeq
Here the complex symmetric matrix $V''(x_j)$ is such that ${\rm Re\,} V''(x_j) \geq 0$, in view of (\ref{eq1.2}).

\medskip
\noindent
Our starting point is the following result, which is a consequence of~\cite[Theorem 1]{HiPr2} and~\cite[Theorem 1.1]{HiPr3}.

\begin{theo}
\label{theo_spectra}
Let us make the assumptions {\rm (\ref{eq1.2})} -- {\rm (\ref{eq1.6})}, {\rm (\ref{eq1.10})}, {\rm (\ref{eq1.11})}, and {\rm (\ref{eq1.12})}. Assume furthermore that the complex symmetric matrix $V''(x_j)$ introduced in {\rm (\ref{eq1.13})} is invertible, for $1\leq j \leq N$. Let $C>0$. Then there exists $h_0>0$ such that for all $h\in (0,h_0]$, the spectrum of the operator $P$ in the open disc $D(0,Ch) = \{z\in \comp; \abs{z} < Ch\}$, is discrete, and the eigenvalues are of the form
\begeq
\label{eq1.14}
\lambda_{j,k} \sim h \left(\mu_{j,k} + h^{1/N_{j,k}} \mu_{j,k,1} + h^{2/N_{j,k}} \mu_{j,k,2} +\ldots\right),\,\, 1\leq j \leq N.
\endeq
Here $\mu_{j,k}$ are the eigenvalues in the disc $D(0,C)$ of the quadratic differential operator
\begeq
\label{eq1.15}
Q_j = \left(D_x - A'(x_j) x\right)^2 + \frac{1}{2} V''(x_j)x \cdot x, \quad 1 \leq j \leq N,
\endeq
repeated with their algebraic multiplicity, and $N_{j,k}$ is the dimension of the corresponding generalized eigenspace. (Possibly after changing $C>0$, we may assume that $\abs{\mu_{j,k}}\neq C$ for all $k$, $1 \leq j \leq N$.) Furthermore, for each $\delta > 0$, the following resolvent estimate holds,
\begeq
\label{eq1.16}
\norm{\left(P - z\right)^{-1}}_{{\cal L}(L^2, L^2)} \leq {\cal O}_{\delta}\left(\frac{1}{h}\right), \quad z\in D(0,Ch),\quad {\rm dist}(z, {\rm Spec}(P))\geq \delta h.
\endeq
\end{theo}

\medskip
\noindent
{\it Remark}. As we shall discuss in Sections \ref{proof} and \ref{spectra}, the invertibility of the complex symmetric matrix $V''(x_j)$ implies that the spectrum of the non-selfadjoint non-elliptic quadratic differential operator $Q_j$ in (\ref{eq1.15}) is discrete and can be determined explicitly. The eigenvalues $\mu_{j,k}$ in (\ref{eq1.14}) satisfy ${\rm Re}\, \mu_{j,k} > 0$ and they are confined to an angular sector of the form $\abs{\arg z} \leq \theta_0 < \pi/2$.

\bigskip
\noindent
The following is the main result of this note.
\begin{theo}
\label{theo_main}
Let $P = (hD_x - A(x))^2 + V(x)$ be such that the assumptions {\rm (\ref{eq1.2})} -- {\rm (\ref{eq1.6})}, {\rm (\ref{eq1.10})}, {\rm (\ref{eq1.11})}, and {\rm (\ref{eq1.12})} hold. Assume that the complex symmetric matrix $V''(x_j)$ in {\rm (\ref{eq1.13})} is invertible, for $1\leq j \leq N$. Recall the quadratic approximations $Q_j$ in {\rm (\ref{eq1.15})} with the eigenvalues $\mu_{j,k}$ and let $a>0$ be such that $a\neq {\rm Re}\, \mu_{j,k}$ for all $k$, $1\leq j \leq N$. Let us introduce the finite set $\Lambda = \{\mu_{j,k}; \, {\rm Re}\, \mu_{j,k} < a\}$ and assume that the eigenvalues $\mu_{j,k}\in \Lambda$ are simple and distinct. Let $\lambda_{j,k} = {\cal O}(h)$ be the eigenvalues of $P$ associated to $\mu_{j,k}\in \Lambda$ according to Theorem {\rm \ref{theo_spectra}}. Then $\lambda_{j,k} = h\left(\mu_{j,k} + {\cal O}(h)\right)$ are also simple and we have in ${\cal L}(L^2(\real^n), L^2(\real^n))$, uniformly as $t\geq 0$, $h\rightarrow 0^+$,
\begeq
\label{eq1.17}
e^{-tP/h} = \sum_{\mu_{j,k}\in \Lambda} e^{-t\lambda_{j,k}/h} \Pi_{\lambda_{j,k}} + {\cal O}(e^{-ta}).
\endeq
Here
\begeq
\label{eq1.18}
\Pi_{\lambda_{j,k}} = \frac{1}{2\pi i } \int_{\abs{z - \lambda_{j,k}} = \varepsilon h} \left(z - P\right)^{-1}\, dz = {\cal O}(1): L^2(\real^n) \rightarrow L^2(\real^n), \quad 0 < \varepsilon \ll 1,
\endeq
is the corresponding rank one spectral projection.
\end{theo}

\medskip
\noindent
{\it Remark}. Theorem \ref{theo_main} is inspired by and closely related to~\cite[Theorem 1.4]{HeSjSt}, where an analogous result is established for a broad class of non-selfadjoint non-elliptic semiclassical pseudodifferential operators with double characteristics, including the Kramers-Fokker-Planck operator. We would therefore like to observe that Theorem \ref{theo_main} above does not seem to follow directly from the analysis in~\cite{HeSjSt}. Indeed, one of the assumptions made in~\cite[Theorem 1.4]{HeSjSt}, namely assumption (H3) there, demands that the principal symbol $p = p_1 + ip_2$ of the operator in question should be such that $\partial^{\alpha} p_1 = {\cal O}(1)$, $\abs{\alpha}\geq 2$. This assumption is not satisfied in the magnetic Schr\"odinger case, for $p$ given by (\ref{eq1.1}) and indeed, a crucial role in our proof of Theorem \ref{theo_main} is played by the recent work~\cite{B2} of the first-named author, where resolvent estimates of subelliptic type are established for the operator $P$ in (\ref{eq1.7}), in an unbounded parabolic neighborhood of the imaginary axis, away from a bounded region near the origin.

\medskip
\noindent
{\it Remark}. Let us assume that $N=1$ in (\ref{eq1.11}) and that $V_1^{-1}(0) = \{0\}$. Associated to the doubly characteristic point $(0,A(0))\in \real^{2n}$ of the symbol $p$ in (\ref{eq1.1}) is the quadratic approximation
\begeq
\label{eq1.19}
Q = \left(D_x - A'(0)x\right)^2 + \frac{1}{2} V''(0)x\cdot x,
\endeq
where ${\rm Re}\, V''(0)\geq 0$, $V''(0)$ invertible. It follows from Lemma \ref{singular} below combined with~\cite[Theorem 2.1]{OtPaPr} that the first eigenvalue $\mu_0\in \comp$ in the bottom of the spectrum of the operator $Q$ has algebraic multiplicity one. The corresponding eigenvalue $\lambda_0$ of the operator $P$ in Theorem \ref{theo_spectra} is then also simple, with
$$
\lambda_0 \sim h \left(\mu_0 + h\mu_{0,1} + \ldots\right),
$$
and an application of Theorem \ref{theo_main} allows us to write in ${\cal L}(L^2(\real^n), L^2(\real^n))$,
$$
e^{-tP/h} = e^{-t\lambda_0/h} \Pi_{\lambda_0} + {\cal O}(e^{-ta}), \quad t \geq 0.
$$
Here $0 < {\rm Re}\, \mu_0 < a < {\rm Re}\, \mu_0 + \tau_0$, with $\tau_0 > 0$ being the spectral gap for the quadratic operator $Q$ in (\ref{eq1.19}), so that
$$
{\rm Spec}(Q)\backslash \{\mu_0\} \subset \{z\in \comp; {\rm Re}\, z \geq {\rm Re}\, \mu_0 + \tau_0\}.
$$

\bigskip
\noindent
Decay estimates for evolution semigroups associated to semiclassical non-selfadjoint Schr\"odinger operators have been studied recently in~\cite{Henry},~\cite{AlHePan},~\cite{AlGrHe}. See also~\cite{Wang}. While the techniques of~\cite{HiPr2},~\cite{HiPr3},~\cite{B1},~\cite{B2}, on which this note is based, are somewhat different, here too we rely essentially on the quantitative version of the Gearhart-Pr\"uss theorem, due to~\cite{HeSj10}.

\bigskip
\noindent
The plan of this note is as follows. In Section 2, after verifying that Theorem \ref{theo_spectra} is a consequence of the results of~\cite{HiPr2},~\cite{HiPr3}, we establish Theorem \ref{theo_main} by combining Theorem \ref{theo_spectra} with the results of~\cite{B1},~\cite{B2}, and~\cite{HeSj10}, see also~\cite{Helffer_book},~\cite{Sj_book}. In Section 3 we discuss quadratic non-selfadjoint magnetic Schr\"odinger operators of the form (\ref{eq1.19}) and obtain an explicit description of their spectra, following~\cite{HiPr1}.

\bigskip
\noindent
{\bf Acknowledgements}. This paper was completed in September 2018, when the second-named author was visiting Universit\'e Paris 13.
It is a great pleasure for him to thank its Mathematics Department for the generous hospitality and excellent working conditions. His best thanks are due especially to Francis Nier for the stimulating discussions.

\section{Resolvent estimates and proof of Theorem \ref{theo_main}}
\label{proof}
\setcounter{equation}{0}
In the beginning of this section, we shall discuss how Theorem \ref{theo_spectra} can be deduced from the results of~\cite{HiPr2},~\cite{HiPr3}. When doing so, let us set
\begeq
\label{eq2.1}
m(X) = 1 + \left(\xi - A(x)\right)^2 + V_1(x) +  \abs{V_2'(x)}^2, \quad X = (x,\xi)\in \real^{2n}.
\endeq
Using (\ref{eq1.2}) and (\ref{eq1.5}) we see that
\begeq
\label{eq2.1.1}
\nabla V_1 = {\cal O}(V_1^{1/2}),
\endeq
and it follows that the function $1\leq m \in C^{\infty}(\real^{2n})$ satisfies $\nabla m = {\cal O}(m^{1/2})$. The following elementary observation
shows that $m$ is then an order function, in the sense that for some $C>0$, we have
$$
m(X) \leq C\langle{X-Y\rangle}^2 m(Y),\quad X,Y\in \real^{2n},
$$
see~\cite{DiSj}. Here $\langle{X\rangle} = \left(1+\abs{X}^2\right)^{1/2}$.

\begin{lemma}
Let $1\leq f\in C^{\infty}(\real^N)$ be such that
\begeq
\label{eq2.2}
\nabla f = {\cal O}(f^{\gamma}),
\endeq
for some $0 < \gamma < 1$. Then there exists $C>0$ such that
\begeq
\label{eq2.3}
f(X) \leq C \langle{X-Y\rangle}^{1/(1-\gamma)} f(Y),\quad X,Y\in \real^N.
\endeq
\end{lemma}
\begin{proof}
Setting
$$
\varphi(t) = f^{1-\gamma}(tX + (1-t)Y), \quad t\in [0,1],
$$
we write, using (\ref{eq2.2}),
\begin{multline*}
\varphi(1) - \varphi(0) = \int_0^1 \varphi'(t)\, dt \\
= (1-\gamma) \int_0^1 f^{-\gamma}(tX+(1-t)Y) \nabla f(tX + (1-t)Y)\cdot (X-Y)\, dt = {\cal O}(\abs{X-Y}).
\end{multline*}
We get, using that $f\geq 1$,
$$
f^{1-\gamma}(X) \leq f^{1-\gamma}(Y) + {\cal O}(\abs{X-Y}) \leq C\langle{X-Y\rangle} f^{1-\gamma}(Y), \quad C > 0,
$$
and (\ref{eq2.3}) follows.
\end{proof}

\medskip
\noindent
{\it Remark}. Using (\ref{eq1.2}), (\ref{eq1.3}), (\ref{eq1.4}), and (\ref{eq1.5}), we may check that more generally, we have $\partial^{\alpha} m = {\cal O}(m^{1/2})$, $\abs{\alpha}\geq 1$.

\medskip
\noindent
Associated to the order function $m$ in (\ref{eq2.1}) is the symbol class $S(m)$ defined by
\begeq
\label{eq2.4}
S(m) = \left\{a\in C^{\infty}(\real^{2n}); \,\, m^{-1}\partial^{\alpha} a \in L^{\infty}(\real^{2n}),\quad \forall \alpha \in \nat^{2n}\right\},
\endeq
and we immediately see, using (\ref{eq1.2}), (\ref{eq1.3}), (\ref{eq1.4}), (\ref{eq1.5}), (\ref{eq1.6}), and (\ref{eq2.1.1}), that $p\in S(m)$, with ${\rm Re}\, p\geq 0$. The assumption (\ref{eq1.10}) implies that ${\rm Re}\, p$ is elliptic at infinity in the sense that for some $\widetilde{C}>1$, we have
\begeq
\label{eq2.5}
{\rm Re}\, p(X) \geq \frac{1}{\widetilde{C}} m(X),\quad X = (x,\xi),\,\, \abs{X} \geq \widetilde{C}.
\endeq
It follows furthermore from (\ref{eq1.11}) that the set $\left({\rm Re}\,p\right)^{-1}(0)\subset \real^{2n}$ is finite,
\begeq
\label{eq2.6}
\left({\rm Re}\,p\right)^{-1}(0) = \left\{X_1,\ldots\, X_N\right\}, \quad X_j = (x_j, A(x_j)),
\endeq
and using (\ref{eq1.12}) we see that the points $X_j$, $\leq j \leq N$, are all doubly characteristic for the symbol $p$. For each $X_j\in \left({\rm Re}\,p\right)^{-1}(0)$, let us define $q_j$ to be the quadratic approximation of $p$ in a neighborhood of $X_j$, so that
\begeq
\label{eq2.7}
p(X_j + Y) = q_j(Y) + {\cal O}(Y^3),\quad Y \rightarrow 0.
\endeq
Using (\ref{eq1.1}) we see that
\begeq
\label{eq2.8}
q_j(Y) = \left(\eta - A'(x_j) y\right)^2 + \frac{1}{2} V''(x_j)y\cdot y,\quad Y = (y,\eta)\in \real^{2n}.
\endeq
Here $A'(x_j)$ is a real $n \times n$ matrix and $V''(x_j)$ is a complex symmetric $n \times n$ matrix such that ${\rm Re}\, V''(x_j)\geq 0$. Following~\cite{HiPr2},~\cite{HiPr3}, we would like to demand that the restriction of the quadratic form $q_j$ to the corresponding singular space should be elliptic, $1\leq j \leq N$. Here, following~\cite{HiPr1},~\cite{HiPrStVi}, we define the singular space $S$ of a complex-valued quadratic form $q$ on $\real^{2n}_{y,\eta}$ with ${\rm Re}\, q \geq 0$, as follows,
\begeq
\label{eq2.9}
S = \left\{Y\in \real^{2n}; H^k_{{\rm Im}\, q} {\rm Re}\, q(Y) = 0,\quad k\in \nat\right\},
\endeq
with $H_f = f'_{\eta}\cdot \partial_y - f'_y \cdot \partial_{\eta}$ being the Hamilton vector field of a function $f\in C^1(\real^{2n}_{y,\eta};\real)$.

\medskip
\noindent
The partial ellipticity property of the quadratic form in (\ref{eq2.8}) along its singular space can be easily characterized explicitly.
\begin{lemma}
\label{singular}
Let
\begeq
\label{eq2.10}
q(y,\eta) = \left(\eta - Ay\right)^2 + \frac{1}{2} Vy\cdot y,\quad (y,\eta) \in \real^{2n},
\endeq
where $A$ is a real $n \times n$ matrix and $V$ is a complex $n \times n$ symmetric matrix such that ${\rm Re}\, V \geq 0$. The restriction of $q$ to the singular space $S$, defined in {\rm (\ref{eq2.9})}, is elliptic precisely when the matrix $V$ is invertible.
\end{lemma}
\begin{proof}
Writing $V = V_1+ i V_2$, we compute, using (\ref{eq2.10}),
$$
H_{{\rm Im}\, q} {\rm Re}\, q(Y) = -2 V_2 y\cdot\left(\eta - Ay\right),
$$
$$
H_{{\rm Im}\, q}^2 {\rm Re}\, q(Y) = 2 \abs{V_2 y}^2,
$$
$$
H_{{\rm Im}\, q}^k {\rm Re}\, q(Y) = 0,\quad k\geq 3,
$$
and therefore the singular space $S$ is given by
\begeq
\label{eq2.11}
S = \left\{Y = (y,Ay)\in \real^{2n}; V_1 y\cdot y = 0,\,\, V_2 y = 0\right\}.
\endeq
We see that the restriction of $q$ to $S$ vanishes, and thus, $q$ is elliptic along its singular space $S$ precisely when $S = \{0\}$. Using (\ref{eq2.11}) and recalling that $V_1 \geq 0$, we see that the latter condition is equivalent to the fact that
\begeq
\label{eq2.12}
{\rm Ker}\,V \cap \real^{n} = \{0\}.
\endeq
To conclude the proof it suffices to notice, or recall from Proposition 4.4 of~\cite{H95}, that ${\rm Ker}\, V \subset \comp^n$ is the complexification of its intersection with $\real^n$.
\end{proof}

\medskip
\noindent
It is now clear, in view of the discussion above, that Theorem \ref{theo_spectra} follows by a direct application of~\cite[Theorem 1]{HiPr2}, \cite[Theorem 1.1]{HiPr3}.

\bigskip
\noindent
We now proceed to discuss the proof of Theorem \ref{theo_main}. When doing so, we observe first that it follows from the other assumptions in Theorem \ref{theo_spectra} that the assumption (\ref{eq1.6}) can be strengthened to the following estimate,
\begeq
\label{eq2.13}
\abs{V_2(x)} \leq {\cal O}(1)\left(V_1(x) + \abs{V_2'(x)}^2\right), \quad x \in \real^n.
\endeq
Indeed, when $x$ is large, (\ref{eq2.13}) follows from (\ref{eq1.6}) and (\ref{eq1.10}), and it suffices therefore to verify (\ref{eq2.13}) in a neighborhood of the (finite) set $V_1^{-1}(0)$ given in (\ref{eq1.11}). When $x_j\in V_1^{-1}(0)$, we write for $y\in {\rm neigh}(0,\real^n)$ small enough,
\begeq
\label{eq2.14}
V_1(x_j + y) + \abs{V_2'(x_j+y)}^2 = \frac{1}{2} V_1''(x_j)y\cdot y + \abs{V_2''(x_j)y}^2 + {\cal O}(y^3) \sim \abs{y}^2,
\endeq
in view of the fact that $V_1''(x_j) \geq 0$ and the matrix $V''(x_j)$ is bijective. Recalling that $V_2(x_j + y) = {\cal O}(y^2)$, in view of (\ref{eq1.12}), we conclude that (\ref{eq2.13}) holds near $x_j$, $1\leq j \leq N$ and hence, globally.

\medskip
\noindent
Thanks to the assumptions (\ref{eq1.2}), (\ref{eq1.3}), (\ref{eq1.4}), (\ref{eq1.5}), and (\ref{eq2.13}), we are now able to apply the results of~\cite{B2} to obtain some precise control on the resolvent of $P$ in a unbounded parabolic neighborhood of the imaginary axis, away from an ${\cal O}(h)$-neighborhood of the origin. It follows from~\cite[Theorem 1.1]{B2} that there exist $h_0 > 0$, $B>0$, and $C>0$ such that for all $h\in (0,h_0]$ and all $z\in \comp$ such that $\abs{z}\geq Ch$, ${\rm Re}\, z \leq B h^{2/3}\abs{z}^{1/3}$, the resolvent $(P-z)^{-1}$ exists and satisfies
\begeq
\label{eq2.14.1}
\norm{\left(P-z\right)^{-1}}_{{\cal L}(L^2,L^2)} \leq {\cal O}(1) h^{-2/3} \abs{z}^{-1/3}.
\endeq

\medskip
\noindent
Recalling the quadratic forms $q_j$ in (\ref{eq2.8}), $1\leq j \leq N$, we let $a>0$ be such that
\begeq
\label{eq2.15}
a\notin \bigcup_{j=1}^N {\rm Re}\left(\, {\rm Spec}\left(q_j^w(x,D_x)\right)\right).
\endeq
We claim that the line ${\rm Re}\, z = ah$ avoids the spectrum of $P$ and that we have
\begeq
\label{eq2.16}
\norm{\left(P-z\right)^{-1}}_{{\cal L}(L^2,L^2)} \leq {\cal O}\left(\frac{1}{h}\right),
\endeq
uniformly along the line ${\rm Re}\, z = ah$, for all $h>0$ small enough. Indeed, when ${\rm Re}\, z = ah$, $\abs{z}\leq Ch$, $C>0$ large enough, it follows from Theorem \ref{theo_spectra} that ${\rm dist}(z,{\rm Spec}(P))\geq h/{\cal O}(1)$, and the bound (\ref{eq2.16}) follows from (\ref{eq1.16}).
In the region where ${\rm Re}\, z = ah$, $\abs{z}\geq Ch$, we may apply the resolvent estimate (\ref{eq2.14.1}), assuming, as we may, that $0 < a < B C^{1/3}$, and the bound (\ref{eq2.16}) follows. All the assumptions of~\cite[Theorem 1.6]{HeSj10} are therefore satisfied, with the operator $A$ there given by $A = -P/h$, and an application of the latter result together with~\cite[Proposition 2.1]{HeSj10} and (\ref{eq2.16}) allows us to write
\begeq
\label{eq2.17}
e^{-tP/h} = e^{-tP/h}\Pi_+ + R(t), \quad t \geq 0.
\endeq
Here $\Pi_+$ is the spectral projection associated with the spectrum of $P$ in the region $\{z\in \comp; {\rm Re}\, z < ah\}$ and
\begeq
\label{eq2.18}
\norm{R(t)}_{{\cal L}(L^2, L^2)} \leq {\cal O}(1) e^{-ta},
\endeq
uniformly as $t\geq 0$, $h\rightarrow 0^+$. Here we have also used the fact that
$$
\Pi_+ = {\cal O}(1): L^2(\real^n) \rightarrow L^2(\real^n).
$$
The proof of Theorem \ref{theo_main} is complete.

\section{Eigenvalues for quadratic magnetic Schr\"odinger operators}
\label{spectra}
\setcounter{equation}{0}
Let $A$ be a real $n \times n$ matrix and let $V= V_1 + iV_2$ be a complex symmetric $n \times n$ matrix such that $V_1 \geq 0$ and $V$ is invertible. In this section, we shall determine the spectrum of a non-selfadjoint quadratic magnetic Schr\"odinger operator of the form
\begeq
\label{eq3.1}
Q = q^w(x,D_x) = \left(D_x - Ax\right)^2 + \frac{1}{2} Vx\cdot x.
\endeq
See also~\cite{MU}. Here when viewing $Q$ a closed densely defined operator on $L^2(\real^n)$, we shall equip it with the maximal domain,
\begeq
\label{eq3.2}
{\cal D}(Q) = \{u\in L^2(\real^n); Qu\in L^2(\real^n)\}.
\endeq
Performing a unitary conjugation of $Q$, we may and shall assume in what follows that the matrix $A$ is antisymmetric, $A^t = -A$.

\medskip
\noindent
Associated to the operator $Q$ is the quadratic Weyl symbol
\begeq
\label{eq3.21}
q(x,\xi) = (\xi - Ax)^2 + \frac{1}{2}V x\cdot x,
\endeq
with ${\rm Re}\, q \geq 0$. In Lemma \ref{singular} we verified that the singular space $S$ of $q$ satisfies $S = \{0\}$, and an application of~\cite[Theorem 1.2.2]{HiPr1} shows that the spectrum of $Q$ is discrete. Continuing to follow~\cite{HiPr1}, we shall now recall its explicit description, and to this end let us introduce the Hamilton map $F$ of $q$,
$$
F: \comp^{2n} \rightarrow \comp^{2n},
$$
given by
\begeq
\label{eq3.3}
FX = \frac{1}{2} H_q(X),\quad X \in \comp^{2n}.
\endeq
Here $H_q$ is the (linear) Hamilton vector field of $q$. A simple computation using (\ref{eq3.21}) and the fact that $A$ is antisymmetric shows that
\begeq
\label{eq3.4}
F = \begin{pmatrix} -A & 1\\
A^2 - V/2 & -A
\end{pmatrix}.
\endeq
The factorization
\begeq
\label{eq3.5}
\left(\lambda - F\right) \begin{pmatrix} 1 & 0\\
A + \lambda & 1
\end{pmatrix} = \begin{pmatrix} 0 & -1\\
T(\lambda) & A + \lambda
\end{pmatrix}, \quad \lambda \in \comp,
\endeq
where
\begeq
\label{eq3.6}
T(\lambda) = \lambda^2 + 2 \lambda A + \frac{1}{2}V,
\endeq
implies that $\lambda_0 \in \comp$ is an eigenvalue of $F$ precisely when ${\rm Ker}\, (T(\lambda_0)) \neq \{0\}$, so that $\lambda_0$ is an eigenvalue of the quadratic $n\times n$ matrix pencil $T(\lambda)$. Furthermore, using (\ref{eq3.5}) we see that the multiplicities agree, so that
\begeq
\label{eq3.7}
{\rm tr}\, \frac{1}{2\pi i}\int_{\abs{\lambda - \lambda_0} = \varepsilon} \left(\lambda - F\right)^{-1}\, d\lambda =
{\rm tr}\, \frac{1}{2\pi i}\int_{\abs{\lambda - \lambda_0} = \varepsilon} T(\lambda)^{-1} \partial_{\lambda} T(\lambda)\, d\lambda, \quad 0 < \varepsilon \ll 1.
\endeq
Indeed, when verifying (\ref{eq3.7}), we may use the following general observation, see~\cite{Sj01}. Let $\Omega \subset \comp$ be open and let $A(\lambda)$, $B(\lambda)$, $C(\lambda)$, $\lambda \in \Omega$, be holomorphic families of invertible $N\times N$ matrices, such that
\begeq
\label{eq3.71}
A(\lambda) = B(\lambda) C(\lambda).
\endeq
Differentiating (\ref{eq3.71}) and using the cyclicity of the trace, we get
\begeq
\label{eq3.72}
{\rm tr}(B^{-1}\partial_{\lambda} B) = {\rm tr}(A^{-1}\partial_{\lambda}A - C^{-1}\partial_{\lambda}C),
\endeq
and applying (\ref{eq3.72}) to the factorization (\ref{eq3.5}), we obtain (\ref{eq3.7}).

\bigskip
\noindent
The matrix $V$ is invertible and it follows therefore from (\ref{eq3.5}), (\ref{eq3.6}) that the Hamilton map $F$ is bijective. More generally, from~\cite{HiPr1},~\cite{OtPaPr} we know that the fact that the singular space of $q$ is trivial implies that the spectrum of $F$ avoids the real axis. This can also be seen directly using (\ref{eq3.5}), (\ref{eq3.6}) and verifying that the equation $T(\lambda) x = 0$ has no non-trivial solutions $x\in \comp^n$ when $\lambda \in \real$. To this end, observing that $Ax\cdot \overline{x}$ is purely imaginary, we write
$$
0 = {\rm Re}\, \left(T(\lambda)x\cdot \overline{x}\right) = \frac{1}{2} V_1 x\cdot \overline{x} + \lambda^2 \abs{x}^2 = 0.
$$
Here $V_1 \geq 0$ and therefore $x = 0$, since $\lambda = 0$ is not an eigenvalue. The relation $T(\lambda)^t = T(-\lambda)$, which holds due to $A$ being antisymmetric, implies the general fact that $\lambda\in \comp$ is an eigenvalue of $F$ precisely when so is $-\lambda$, and using the definition of the multiplicity given by the right hand side of (\ref{eq3.7}) we also recover the fact that the multiplicities agree, see~\cite[Section 21.5]{H_book}. Let $\lambda_1,\ldots, \lambda_n$ be the eigenvalues of $F$, counted according to their multiplicity, such that ${\rm Re}\, (\lambda_j/i) > 0$. The following result now follows from~\cite[Theorem 1.2.2]{HiPr1} and the discussion above.

\begin{theo}
Let $A$ be a real $n\times n$ matrix and let $V$ be a complex symmetric $n \times n$ matrix such that ${\rm Re}\, V\geq 0$ and $V$ is invertible. The spectrum of the quadratic operator $Q = q^w(x,D_x)$ in {\rm (\ref{eq3.1})} is discrete and is given by the eigenvalues of the form
$$
\sum_{j=1}^n \frac{\lambda_j}{i} \left(1 + 2\nu_{j,\ell}\right), \quad \nu_{j,\ell}\in \nat.
$$
Here $\lambda_j \in \comp$ are such that ${\rm Im}\, \lambda_j > 0$ and ${\rm Ker}\, (T(\lambda_j)) \neq \{0\}$, with the quadratic matrix pencil $T(\lambda)$ given in {\rm (\ref{eq3.6})}.
\end{theo}

\end{document}